\documentclass{amsart}
\usepackage{graphicx}
\usepackage{amssymb,amscd,amsthm,amsxtra}
\usepackage{latexsym}
\usepackage{epsfig}
\usepackage{mathtools}
\usepackage{esint}
\usepackage{color}

\newtheorem{thm}{Theorem}[section]
\newtheorem{cor}[thm]{Corollary}
\newtheorem{lem}[thm]{Lemma}
\newtheorem{prop}[thm]{Proposition}
\theoremstyle{definition}

\theoremstyle{remark}
\newtheorem{rem}[thm]{Remark}
\numberwithin{equation}{section}
\newcommand{\be}{\begin{equation}}
\newcommand{\ee}{\end{equation}}

\newcommand{\R}{\mathbb R}

\newcommand{\eps}{\epsilon}

\newcommand{\p}{\partial}

\newcommand{\comment}[1]{}

\begin{document}

\title[Uniform density estimates]{Uniform density estimates and $\Gamma$-convergence for the Alt-Phillips functional of negative powers}
\author{D. De Silva}
\address{Department of Mathematics, Barnard College, Columbia University, New York, NY 10027}
\email{\tt  desilva@math.columbia.edu}
\author{O. Savin}
\address{Department of Mathematics, Columbia University, New York, NY 10027}\email{\tt  savin@math.columbia.edu}
\begin{abstract} We obtain density estimates for the free boundaries of minimizers $u \ge 0$ of the Alt-Phillips functional involving negative power potentials
 $$\int_\Omega \left(|\nabla u|^2 + u^{-\gamma} \chi_{\{u>0\}}\right) \, dx, \quad \quad \gamma \in (0,2).$$
 These estimates remain uniform as the parameter $\gamma \to 2$. As a consequence we establish the uniform convergence of the corresponding free boundaries to a minimal surface as $\gamma \to 2$.  
 The results are based on the $\Gamma$-convergence of these energies (properly rescaled) to the Dirichlet-perimeter functional 
 $$\int_\Omega |\nabla u|^2 dx + Per_{\Omega}(\{ u=0\}),$$
 considered by Athanasopoulous, Caffarelli, Kenig, and Salsa.
 \end{abstract}

\maketitle

\section{Introduction}

Energy functionals involving the Dirichlet integral of a density $u$ and a potential term $W(u)$
$$ \int_\Omega |\nabla u|^2 + W(u) \, dx,$$
appear in various models in the calculus of variations. A classical example is the Allen-Cahn \cite{AC} energy given by the double-well potential $$W(t)=(1-t^2)^2,$$ which is relevant in the theory of phase-transitions and minimal surfaces. In their celebrated result, Modica and Mortola \cite{MM} showed that $0$-homogenous rescalings of bounded minimizers $|u| \le 1$, converge up to subsequences to a $\pm 1$ configuration separated by a minimal surface, i.e.
\begin{equation}\label{ueps}
u_\eps(x)= u\left( \frac x \eps \right)  \quad \to \quad \chi_E - \chi_{E^c}  \quad \mbox{in $L^1_{loc}$, \quad as $\eps \to 0$,}
\end{equation}
with $E$ a set of minimal perimeter. At the level of the energy, this result is expressed in terms of the Gamma-convergence of the rescaled energies
$$ \int_\Omega \eps |\nabla u|^2 + \frac 1 \eps W(u) \, dx,$$
to a multiple of the perimeter functional $c_0 Per_{\Omega}(E)$.

Other examples of energies appear in the theory of free boundary problems. 
When the potential $W(t)$ is not of class $C^{1,1}$ near a minimum point, say $t=0$, minimizers can develop patches where they take this value. 
The boundary of such a patch $\p\{u=0\}$ is the free boundary. Two particular potentials of interest are given by $$W(t)=t^+,$$ which corresponds to the obstacle problem (for a comprehensive survey see \cite{PSU}), and by $$W(t) = \chi_{\{t>0\}},$$ which corresponds to the Bernoulli free boundary problem (see for example \cite{AC,ACF,CS}). These can be viewed as part of the family of power-potentials
$$W(t)=(t^+)^\beta, \quad \quad \beta\in [0,2),$$
which were considered by Alt and Phillips \cite{AP} in the early 80's.

Recently in \cite{DS}, we investigated properties of non-negative minimizers and their free boundaries for Alt-Phillips potentials of negative powers
$$W(t)=t^{-\gamma} \chi_{\{t>0\}}, \quad \quad \gamma \in (0,2).$$
These potentials are relevant in the applications, for example in liquid models with large cohesive internal forces in regions of low density. 
The upper bound $\gamma <2$ is necessary for the finiteness of the energy.

In \cite{DS} we showed that minimizers $u \ge 0$ of the Alt-Phillips functional involving negative power potentials
 \begin{equation}\label{Ef}
\mathcal E_\gamma(u):=\int_\Omega \left(|\nabla u|^2 + u^{-\gamma} \chi_{\{u>0\}}\right) \, dx, \quad \quad \gamma \in (0,2),
 \end{equation}
have optimal $C^\alpha$ H\"older continuity. The free boundary $$F(u):=\p \{u>0\}$$
is characterized by an expansion of the type
$$ u= c_\alpha d^ \alpha + o(d^{2-\alpha}), \quad \quad \alpha:=\frac{2}{2+\gamma}\quad  \in (\frac 12, 1),$$ 
where $d$ denotes the distance to $F(u)$ and $c_\alpha d^\alpha$ represents the explicit 1D homogenous solution. Furthermore, we showed that $F(u)$ is a hypersurface of class $C^{1,\beta}$ 
up to a closed singular set of dimension at most $n- k(\gamma)$, where $k(\gamma) \ge 3$ is the first dimension in which a nontrivial $\alpha$-homogenous minimizer exists. We also established the Gamma-convergence of a suitable multiple of the $\mathcal E_\gamma$ to the perimeter of the positivity set $Per_\Omega(\{u>0\})$ as $\gamma \to 2$.

In this work we investigate in more detail the properties of minimizers as the parameter $\gamma$ tends to the critical value $2$, 
and make precise the connection between their free boundaries and the theory of minimal surfaces. 
In particular we establish density estimates and the uniform convergence (up to subsequences) of the free boundaries $F(u_k)$ to a minimal surface, for a sequence of bounded minimizers 
$u_k$ corresponding to parameters $\gamma_k \to 2$, see Corollary \ref{C1}. Uniform convergence results in different settings were obtained by Caffarelli and Cordoba \cite{CC} for the Allen-Cahn energy and the convergence in \eqref{ueps}, and by Caffarelli and Valdinoci \cite{CV} for the $s$-nonlocal minimal surfaces with $s \to 1$. We also refer the reader to other related works in similar contexts \cite{ADM,DFPV,PV,SV,V}. 

The constants in the H\"older and density estimates obtained in \cite{DS} degenerate as $\gamma \to 2$. However, here we develop uniform estimates in $\gamma$, and for this it is convenient to rescale the potential term in the functional $\mathcal E_\gamma$ in a suitable way (see \eqref{Jf}). We further establish the Gamma-convergence to the Dirichlet-perimeter functional 
 $$\mathcal F(u):=\int_\Omega |\nabla u|^2 dx + Per_{\Omega}(\{ u=0\}),$$
which was studied by Athanasopoulous, Caffarelli, Kenig, Salsa in \cite{ACKS}. Heuristically, this shows that the cohesive term $W$ has the effect of surface tension as $\gamma \to 2$.

\section{Main results}

Let $\Omega$ be a bounded domain in $\R^n$ with Lipschitz boundary. We consider $J_\gamma$, a rescaling of $\mathcal E_\gamma$, which acts on functions
$$u : \Omega \to \R, \quad \quad u \in H^1(\Omega),  \quad u \ge 0,$$ and it is defined as
\begin{equation}\label{Jf}
J_\gamma (u,\Omega):=\int _\Omega |\nabla u|^2 + W_\gamma(u) \, \, dx, 
\end{equation}
where
\begin{equation}\label{1Dp00}
W_\gamma (u):= c_\gamma u^{-\gamma} \chi_{\{u>0\}}, \quad \quad  \mbox{with} \quad  c_\gamma:= \frac {1}{16} \cdot (2-\gamma)^{2}, \quad \gamma \in (0,2).
\end{equation} 
We study uniform properties of the minimizers of $J_\gamma$ as $\gamma \to 2^-$. We often drop the dependence on $\gamma$ from $J$ and $W$ when there is no possibility of confusion.

Notice that $u$ is a minimizer of $\mathcal E_\gamma$ defined in \eqref{Ef}, if and only if $c(\gamma) u$ is a minimizer of $J_\gamma$, with $c(\gamma)=c_\gamma^\frac{1}{\gamma+2}$ an appropriate constant depending only on $\gamma$, and $c(\gamma) \to 0$ as $\gamma \to 2$.

The constant $c_\gamma$ in \eqref{1Dp00} is chosen such that
\begin{equation}\label{1Dp0}
\int_0^1 2 \sqrt{W_\gamma(s)} \, \, ds =1.
\end{equation}

The homogenous 1D solution $\varphi$ plays an important role in the analysis. It is given by
\begin{equation}\label{1Dp}
 \varphi(t):=c^*_\gamma \, (t^+)^\alpha,
 \end{equation}
with
$$ \alpha:=\frac{2}{2+\gamma}, \quad \quad \quad c_\gamma^*:=\left((1+ \frac \gamma2)^2 \, c_\gamma \right)^\frac{1}{\gamma+2},$$
and satisfies
\begin{equation}\label{1Dp2}
\varphi '= (W_\gamma(\varphi))^{1/2}, \quad \quad \mbox{in} \quad \{\varphi>0\}.
\end{equation}
We differentiate the last equality and obtain that $\varphi$ solves the Euler-Lagrange equation
\begin{equation}\label{1Dp3}
2\varphi'' = W_\gamma'(\varphi) \quad \quad \mbox{in} \quad \{\varphi>0\}.
\end{equation}

\
 
{\it Positive constants depending only on the dimension $n$ are denoted by $c$, $C$, and referred to as universal constants.}

\

The first result is an optimal uniform growth estimate.

\begin{thm}\label{P1}
Let $u$ be a minimizer of $J_\gamma$ in $B_1$ and assume $u(0)=0.$ Then, there exists a universal constant $C$ such that
$$ u(x) \le C |x|^\alpha, \quad \quad \alpha:=\frac{2}{2+\gamma}, \quad \quad \forall \, x \in B_{1/2}.$$
\end{thm}

The second theorem gives the uniform density estimate of the free boundary.
 
 \begin{thm}[Density estimates]\label{T0} There exists a universal constant $c_0$ such that if $u$ is a nonnegative minimizer of $J_\gamma$ in $B_1$ and $0 \in F(u)$ then
 $$1-c_0 \ge  \frac{|\{u>0\} \cap B_r|}{|B_r|} \ge c_0, \quad \quad \forall \,r \le \frac 1 2.$$

 \end{thm}

The following result is a direct consequence of Theorems \ref{P1} and \ref{T0}.

\begin{cor}\label{C0}
Let $u$ be a nonnegative minimizer of $J_\gamma$ in $B_1$. If $0 \in F(u)$ then for all $r \in (0, 1/2)$ each of the sets $\{u=0\} \cap B_r$ and $\{u>0\} \cap B_r$ contains an interior ball of radius $c r$. Moreover
$$c r^{n-\alpha \gamma} \le J(u,B_r) \le C r^{n-\alpha \gamma}.$$
\end{cor}
Next we introduce the Dirichlet-perimeter functional $\mathcal F$ introduced by Athanasopoulous, Caffarelli, Kenig, Salsa in \cite{ACKS}. It acts on the space of admissible pairs $(u,E)$ consisting of functions $u \ge 0$ and measurable sets $E \subset \Omega$ which have the property that $u=0$ a.e. on $E$,
$$ \mathcal A(\Omega):=\{(u,E)| \quad u \in H^1(\Omega),\quad \mbox{$E$ Caccioppoli set, $u \ge 0$ in $\Omega$, $u=0$ a.e. in $E$} \}.$$
The functional $\mathcal F$ is given by the Dirichlet - perimeter energy
$$\mathcal F_\Omega(u,E)= \int_\Omega |\nabla u|^2 dx + P_\Omega(E),$$
 where $P_{\Omega}(E)$ represents the perimeter of $E$ in $\Omega$
 \begin{align*}
 P_\Omega(E)&=\int_\Omega |\nabla \chi_E| \\
 &=\sup \,  \int_\Omega \chi_E \, div \, g \, dx \quad \mbox{with} \quad g \in C_0^\infty(\Omega), \quad |g| \le 1.  
 \end{align*}

The next theorem establishes the $\Gamma$-convergence of the $J_\gamma$'s.

\begin{thm}\label{Tg} As $\gamma \to 2$, the functionals $J_\gamma$ $\Gamma$-converge to $\mathcal F$. 

More precisely we have:

a) (lower semicontinuity) if $\gamma_k \to 2$ and $u_k$ satisfy
  $$ u_k^{1-\gamma_k/2} \to \chi_{E^c} \quad \mbox{in $L^1(\Omega)$}, \quad u_k \to u \quad \mbox{in $L^2(\Omega)$},$$
  then
  $$\liminf J_{\gamma_k}(u_k,\Omega) \ge \mathcal F_{\Omega}(u,E).$$
  
 b)  (approximation) given $(u,E) \in \mathcal A(\Omega)$ with $u$ a continuous in $\overline \Omega$, there exists $\gamma_k \to 2$ and $u_k$ such that
 $$ u_k^{1-\gamma_k/2} \to \chi_{E^c} \quad \mbox{in $L^1(\Omega)$}, \quad u_k \to u \quad \mbox{in $L^2(\Omega)$},$$
 $$ J_{\gamma_k}(u_k,\Omega) \to  \mathcal F_{\Omega}(u,E).$$
\end{thm}

Our main result gives the strong convergence of the minimizers of $J_\gamma$ and their zero set to the minimizing pairs $(u,E)$ of $\mathcal F$. 
 \begin{thm}\label{TM}
 Let $\Omega$ be a bounded domain with Lipschitz boundary, $\gamma_k \to 2^-$, and $u_k$ a sequence of functions with uniform bounded energies
 $$ \|u_k\|_{L^2(\Omega)} + J_{\gamma_k}(u_k,\Omega) \le M,$$
 for some $M>0$. Then, after passing to a subsequence, we can find $(u,E) \in \mathcal A(\Omega) $ such that
   $$ u_k^{1-\gamma_k/2} \to \chi_{E^c} \quad \mbox{in $L^1(\Omega)$}, \quad u_k \to u \quad \mbox{in $L^2(\Omega)$},$$
   and
 $$  \chi_{\{u_k>0\}} \to \chi_{E^c} \quad \mbox{in $L^1(\Omega)$}.$$
   
  Moreover, if $u_k$ are minimizers of $J_{\gamma_k}$ then the limit $(u,E)$ is a minimizer of $\mathcal F$. The convergence of $u_k$ to $u$ and respectively of the free boundaries $\p \{u_k>0\}$ to $\p E$ is uniform on compact sets  (in the Hausdorff distance sense). 
 \end{thm}

As a consequence we obtain the connection between bounded minimizers of $\mathcal E_\gamma$ with $\gamma \to 2$ and minimal surfaces, as stated in the Introduction. The uniform boundedness of minimizers can be deduced for example from a uniform bound of the boundary data on $\p \Omega$.

\begin{cor}\label{C1}
Assume that $u_k$ are uniformly bounded minimizers of $\mathcal E_{\gamma_k}$ defined in \eqref{Ef}, and $\gamma_k \to 2$. Then, up to subsequences, $F(u_k)$ converge uniformly on compact sets to a minimal surface $\p E$. 
\end{cor}

Indeed, $c(\gamma_k) u_k$ is a minimizer for $J_{\gamma_k}$ and, since $c(\gamma_k) \to 0$, the limiting function $u$ of Theorem \ref{TM} is identically $0$. This means that the limiting set $E$ must be a set of minimal perimeter in $\Omega$.

The paper is organized as follows. In Section 3 we prove the uniform growth estimate Theorem \ref{P1} and in Section 4 we obtain the uniform density estimates. In the last section we prove the main result Theorem \ref{TM}. 

\section{Proof of Theorem \ref{P1}}

In this section we prove Theorem \ref{P1}. We state it here again for the reader convenience. We remark that this statement was proved in \cite{DS}
 with a constant $C$ depending on $\gamma$. The purpose of this section is to show that in fact the statement holds with a universal constant $C$. In the proof, we use that minimizers are viscosity solution in the sense of Definition 4.1 of \cite{DS}, as showed in Proposition 4.4 of \cite{DS}.
 
 \begin{thm}
Let $u$ be a minimizer of $J_\gamma$ in $B_1$, and assume $u(0)=0.$ Then 
$$ u(x) \le C |x|^\alpha, \quad \quad \quad \forall \, x \in B_{1/2},$$
with $C$ universal.
\end{thm}

\begin{proof} Minimizers of $J$ are invariant under $\alpha$-homogenous rescalings
$$ \tilde u(x)=\frac{ u(y_0+\lambda x)}{\lambda^\alpha}.$$
After such a rescaling, we may assume that we are in the situation $B_1\subset \{u>0\}$ and $u$ vanishes at some point $x_0 \in \p B_1$. We need to prove that $u(0)$ is bounded above by a large universal constant.

Notice that in $B_1$ we satisfy
$$ \triangle u \le 0  , \quad \quad \triangle \, (u-1)^+ \ge -1.$$
Thus, if $$u(0) \ge M \gg 1,$$ then by the weak Harnack inequality we find 
$$ u \ge c \, M \quad \mbox{in} \quad B_{1/2}, \quad \mbox{with $c>0$ universal.}$$

\begin{lem}\label{lem1} There exists a one dimensional increasing function $\psi$, 
$$\psi :[0,t_0] \to R, \quad \quad \psi(0)=0, \quad t_0 \le \frac 1 4,$$ such that (see \eqref{1Dp} for the definition of $\varphi$)

1) $$\psi (t)= \varphi(t) + \eps \, t^{2-\alpha} + O(t^{2-\alpha+\delta}) \quad \quad \mbox{near $0$,}\quad \delta>0,$$

2) $$ 2 \psi'' \ge 4 n \psi' + W' (\psi),$$

3) $$\psi(t_0) \le 1 , \quad \quad \psi'(t_0) \le C_0 \quad \quad \mbox{universal}.$$

\end{lem}

\

Using Lemma \ref{lem1} we construct  a barrier $\Psi : B_1 \setminus B_{1/2} \to \R$, as 
$$\Psi (x)= \psi(1-|x|) \quad \mbox{in} \quad  B_1 \setminus B_{ 1-t_0}$$ and 
$$ \triangle \Psi=0 \quad \mbox{in } \quad B_{ 1-t_0} \setminus B_{1/2},$$ 
with boundary conditions
$$\Psi=cM \quad \mbox{on $\p B_{1/2}$}, \quad  \Psi=\psi(t_0) \quad \mbox{on $\p B_{1-t_0}$}.$$
Since $\psi(t_0) \le 1$, it follows that 
$$\mbox{$|\nabla \Psi| >C_0$ in the annulus $B_{ 1-t_0} \setminus B_{1/2}$,}$$
provided that $M$ is large universal.

We claim that
\begin{equation}\label{1000}
2\triangle \Psi \ge W'(\Psi), \quad \mbox{in} \quad B_1 \setminus B_{1/2}.
\end{equation} 


The inequality is satisfied in the outer annulus $B_1 \setminus B_{1-t_0}$ by property 2) above, and in the inner annulus 
$B_{ 1-t_0} \setminus B_{1/2}$ since $0 > W'$. 

Moreover, the inequalities between the normal derivatives on either side of $\p B_{ 1-t_0} $ guarantee that \eqref{1000} holds in the whole domain. 

Since $W'(t)$ is increasing for $t>0$, we can apply the maximum principle and conclude that 
$$u \ge \Psi \quad \mbox{in} \quad B_1 \setminus B_{1/2}.$$ We contradict the free boundary condition at the point $x_0 \in F(u)$ for a minimizer, see Proposition 4.4 in \cite{DS}. Indeed, property 1) above shows that $\Psi-\varphi(d)$ has a positive correction term $\eps \, d^{2-\alpha}$ in the expansion near its free boundary and therefore it is a strict viscosity subsolution on $\p B_1$, see Definition 4.1 in \cite{DS}. 
 \end{proof}

It remains to prove the lemma above.

\

{\it Proof of Lemma \ref{lem1}:} We reduce the second order ODE to a 1st order ODE by taking $\psi$ as an independent variable. More precisely, with a strictly increasing function $\psi$ we associate the function $g>0$ defined on the range of $\psi$ as
\begin{equation}\label{gred}
 g(\psi):=(\psi')^2.
 \end{equation}
After differentiation we obtain 
$$2 \psi'' = g'(\psi).$$ The function $\psi$ can be recovered from $g$ by the formula
\begin{equation}\label{inv}
\psi(t)=G^{-1}(t), \quad \quad G(r):=\int_0^r \frac{1}{ \sqrt {g(s)}} ds.
\end{equation}
In the case when $\psi$ coincides with the 1D solution $\varphi$ given in \eqref{1Dp}, then the associated function $g$ equals $W$, see \eqref{1Dp2}.

In our setting we define $g$ explicitly as
$$ g(s):= W(s) + \bar \eps + C_1 s^{1-\frac \gamma 2}, \quad s \in (0, s_0], $$
with $C_1=8 \, n$ universal, and $s_0$ given by the solution to $$C_1 s^{1-\frac \gamma 2} = c_\gamma s^{-\gamma}=W(s) \quad \mbox{when} \quad s=s_0,$$
and  $\bar \eps>0$ arbitrarily small. Notice that $s_0 \to 0$ as $\gamma \to 2^-$, and from the formula for $c_\gamma$ in \eqref{1Dp0} it follows
 \begin{equation}\label{1001}
 s_0 \sim 2-\gamma.
 \end{equation} Notice that
$$g(s_0) \le 3 C_1s_0^{{1-\frac \gamma 2}}\le 3 C_1=: C_0^2.$$
This gives property 3) since $$\psi(t_0)=s_0, \quad \quad \psi'(t_0)=(g(s_0))^{1/2}.$$
By construction $g \ge W$ which by \eqref{inv} implies $\psi \ge \varphi$. Thus $$s_0 =\psi(t_0) \ge \varphi (t_0),$$ 
and by \eqref{1Dp}, \eqref{1001}, it follows that also $t_0 \to 0$ as $\gamma \to 2^-$. 

We compute
$$ g' = W' + C_1 (1-\frac \gamma 2) s^{-\frac \gamma 2}$$
and use the inequality 
$$C_1 (1-\frac \gamma 2) \ge 8n \sqrt {c_\gamma},$$ 
and that $g \le 3 W$ in the interval $[0,s_0]$ to obtain
$$g' \ge W' + 4n (3 W)^{1/2} \ge W' + 4n g^{1/2},$$
which gives 2). 

Finally, we obtain property 1) from \eqref{inv} and the expansion
$$ \frac{1}{ \sqrt {g(s)}}= \frac{1}{ \sqrt {W(s)}} \left( 1- c(\gamma) \bar \eps s^\gamma + O(s^{1+ \frac \gamma 2})\right).$$

\qed

 \section{Density estimates}
 
 In this technical section we prove Theorem \ref{T0} and Corollary \ref{C0}. We follow the classical ideas from the minimal surface theory by constructing appropriate competitors for the minimizer $u$, and then make use of the isoperimetric inequality. They allows us to obtain discrete differential inequalities for the measure of the sets $\{u>0\}$ in $B_r$, which give the desired conclusion after iteration.
 
 We start with the lower bound.

 \begin{lem} \label{L1} Let $u$ be a minimizer of $J_\gamma$ in $B_1$ and assume $0\in F(u).$ Then 
 $$|\{u>0 \cap B_r| \ge c_0 |B_r|.$$
 \end{lem}
 
 \begin{proof} After a dilation, assume $u$ minimizes $J$ in $B_3$. Since $0 \in F(u)$, 
 $$u \le C_0 \quad \mbox{ in} \quad  B_{2},$$ 
 by Theorem \ref{P1}.
 Define $$ A_r:= \{ u>0\} \cap B_r, \quad \quad a(r):=|A_r|.$$
 It suffices to show that  $$a(1) \le c_0 \quad \Longrightarrow \quad a(r)=0 \quad \mbox{ for all $r$ sufficiently small,}$$ 
 which is not possible since $0 \in F(u)$.
We consider the case when $\gamma$ is close to 2.

 Define $s_0$, $t_0$, as
 \begin{equation}\label{s_0}
 W(s_0)=1, \quad \varphi(t_0)=s_0   \quad \Longrightarrow \quad \varphi '(t_0) =\sqrt{W(s_0)}=1
  \end{equation}
 and notice that $s_0, t_0 \to 0$ as $\gamma \to 2^-$. 
 
 \
 
 {\it Step 1:} We show that the densities of the sets $A_r$ in $B_r$ decay geometrically as we rescale by a factor of $1 - 2 t_0$, i.e if $a(1) \le c_0$ then
 \begin{equation}\label{c1}
  r_0^{-n} a(r_0) \le  r_0 a(1) \quad \mbox{with} \quad r_0:=1-2 t_0.
  \end{equation}
  
  \
 
 First we construct a 1D function.
 \begin{lem}\label{lem2}
 There exists a piecewise $C^1$ function $\psi$ in $[0,1]$ such that
 
 1) $$\psi(t)=\varphi(t) \quad \quad \mbox{ if $ t \le t_0$,}$$
 
 2) $$2 \psi'' + 4n \psi' \le W'(\psi)   \quad \mbox{ if $t \ge t_0$,}$$
 
 3) $$\psi (1) \ge 2 C_0, \quad \psi'(t_0) \le C_1 \quad \mbox{ for some $C_1$ large universal. }$$
 
 \end{lem}
 
 \begin{proof}
 Indeed, we may take
 $$\psi:=\varphi + K g(t) \, \chi_{\{t \ge t_0\}}$$
 with $g$ an increasing $C^2$ function in $[t_0,1]$ such that 
 $$g(t_0)=0, \quad g'' + 2n g' \le -c \mbox{ in $[t_0,1]$,}$$
 and $K$ a sufficiently large universal constant. Properties 1), 3) follow immediately from the definition of $\psi$. For 2) we use that in $[t_0,1]$
 $$2\varphi''= W'(\varphi), \quad \varphi'\le \varphi'(t_0)=1,$$
 hence
 \begin{align*}
 2 \psi'' + 4n \psi'  & \le 2 \varphi '' + 4n \varphi ' - 2Kc \\
  & \le  W'(\varphi) + 4n - 2Kc \\
 & \le W'(\varphi) \\
 & \le W'(\psi),
 \end{align*}
 where in the last inequality we used that $W'$ is an increasing function.
 
 \end{proof}
 \
 
  {\it Proof of Step 1:}
 We use Lemma \ref{lem2} to define $$\Psi(x):=\psi(|x|-(1-t_0)),$$
 and let $$D:= \{ u > \Psi \} \, \subset \quad B_{2-t_0} \cap \{u>0\}.$$
 Notice that $u=\Psi$ on $\partial D$, hence the minimality of $J$ implies
 \begin{equation}\label{101}
 J(u,D) \le J (\Psi, D).
 \end{equation}
 We decompose $D$ as the disjoint union
 $$ D=D_1 \cup D_2, \quad \quad D_1:= D \cap B_{1}, \quad D_2:=D \setminus B_{1},$$ 
and notice that
\begin{align}\label{D2in} & \quad \quad \quad J(\Psi,D_2) - J(u, D_2)=\\
\nonumber &=\int_{D_2} -2\nabla (u-\Psi) \cdot \nabla \Psi - |\nabla (u-\Psi)| ^2 + W(\Psi)-W(u) dx  \\
\nonumber & \le \int_{D_2} (u-\Psi) 2\triangle \Psi + W(\Psi) - W(u) dx + \int_{\partial D_2}2 (u-\Psi) |\nabla \Psi| d \sigma \\
\nonumber & \le \int_{D_2} (u-\Psi) W'(\Psi)+ W(\Psi) - W(u) dx + \int_{\partial D_2 \cap \partial B_1}2 (u-\Psi) |\nabla \Psi| d \sigma\\
\nonumber & \le C \mathcal H^{n-1} (\{u>0\} \cap \partial B_{1}),
\end{align}
where we have used that 
$$0 \le u-\Psi \le C, \quad 2 \triangle \Psi \le W'(\Psi), \quad |\nabla \Psi| \le C \quad \mbox{on $\partial B_1$} ,$$
and that $W$ is convex on its positivity set.

Combining \eqref{101} and \eqref{D2in} we find
\begin{equation}\label{102}
J(u,D_1) \le J(\Psi,D_1) + C \mathcal H^{n-1} (\{u>0\} \cap \partial B_{1}).
\end{equation}
In $D_1$ we use the Cauchy-Schwartz inequality and the coarea formula to obtain
\begin{align}\label{1023}
\nonumber J(u,D_1) &\ge \int_{D_1 \cap \{u< \varphi(t_0) \}} 2|\nabla u| \sqrt{W(u)} dx \\
& = \int_0^{s_0} \mathcal H^{n-1}(\{u=s\}\cap D_1) \, 2 \sqrt {W(s)}ds.
\end{align} 
On the other hand $|\nabla \Psi|=\sqrt{W(\Psi)}$ in $D_1$ by construction (see 1) in Lemma \ref{lem2} and \eqref{1Dp2}) and the inequality above becomes an equality for $\Psi$:
\begin{equation}\label{1025}
J(\Psi,D_1) = \int_0^{s_0} \mathcal H^{n-1}(\{\Psi=s\} \cap D_1) \, 2 \sqrt {W(s)} ds.
\end{equation}
Next we use that $$\{u>s\} \cap B_{1-2t_0} \subset D_1 \cap \{ u>s > \Psi \}, \quad \quad s>0,$$ and the isoperimetric inequality implies
$$c_n |\{u>s\} \cap B_{1-2t_0}|^\frac{n-1}{n} \le \mathcal H^{n-1}(\{u=s\}\cap D_1) + \mathcal H^{n-1}(\{\Psi=s\} \cap D_1),$$
hence
\begin{equation}\label{103}
\int_0^{s_0}  c_n |\{u>s\} \cap B_{1-2t_0}|^\frac{n-1}{n} \, 2 \sqrt {W(s)}  ds \le J(u,D_1) + J(\Psi, D_1).
\end{equation}
We combine this with \eqref{102}, \eqref{1025} and use that 
\begin{equation}\label{1031}
 \int_{B_{1-t_0}} W(u) dx \le J(u,D_1),
 \end{equation} and obtain
\begin{align}\label{enin}
\nonumber \int_{B_{1-2t_0}} W(u) dx \, + \, & \int_0^{s_0}   c_n  |\{u>s\} \cap B_{1-2t_0}|^\frac{n-1}{n} \, 2 \sqrt {W(s)} ds \le \\
\le & \quad C \mathcal H^{n-1} ( \{\Psi=s_0\} \cap  \{u>0\} ) + \\
\nonumber  & + C \int_0^{s_0} \mathcal H^{n-1}(\{\Psi=s\}\cap \{u>0\}) \, 2 \sqrt {W(s)} ds.
 \end{align}
The inequality holds also when we replace $\Psi$ by $\Psi_t$ defined as 
$$\Psi_s(x):=\psi(|x|-(1-t_0-t)), \quad \quad t \in [0,t_0].$$
Notice that $\Psi_0=\Psi$ and $\{\Psi_t = s \}$ is the sphere at distance $t$ from the sphere $\{\Psi=s\}$. Thus, if we write the inequality above for $t \in [0,t_0]$ and average it over this interval we obtain
$$\int_{B_{1-2t_0}} W(u) dx + \int_0^{s_0}   c_n  |\{u>s\} \cap B_{1-2t_0}|^\frac{n-1}{n} \, 2 \sqrt {W(s)} ds \le $$
\begin{equation}\label{mi}
 \le C t_0^{-1} |\{ u>0\} \cap (B_{1} \setminus B_{1-2t_0})|  \int_0^{s_0} \, 2 \sqrt {W(s)} ds.
  \end{equation}
Let $s_1 \in [0, s_0]$ and denote by
$$b:=|\{ 0< u \le s_1\} \cap B_{1-2t_0}|,$$
hence if $s \le s_1$ then
$$|\{u>s\} \cap B_{1-2t_0}| \le |\{u>s_1\} \cap B_{1-2t_0}|=a(1-2t_0) - b.  $$ 
Notice that by the choice of $c_\gamma$ we have
$$ \int_0^{s_1} 2 \sqrt {W(s)} ds = s_1^{1- \frac{\gamma}{2}}.$$
Since $W(u) \ge W(s_1)$ in the set $\{0<u\le s_1\}$, we can bound below the left hand side in \eqref{mi} by
\begin{equation}\label{lhs}
W(s_1) b + c_1 \, s_1^{1-\frac \gamma 2}(a(1-2t_0)-b)^\frac{n-1}{n},
\end{equation}
while the right hand side in \eqref{mi} is bounded above by
$$C_2 \frac{a(1)-a(1-2t_0)}{2 t_0},$$
with $C_2$, $c_1$ universal constants. 

We choose $s_1$ such that $$W(s_1)=C_3 \gg C_2, \quad \quad \mbox{i.e.} \quad s_1 = C_3^{- \frac 1 \gamma} s_0=(c_\gamma/C_3)^\frac 1 \gamma.$$
Using that $c_\gamma \sim (2-\gamma)^2$ we find that the coefficient 
$$c_1 \,  s_1^{1-\frac \gamma 2} $$
which appears in \eqref{lhs} remains bounded below as $\gamma \to 2^-$. This means that if $a(1-2t_0) \le a(1) \le c_0$ small, universal, then the expression in \eqref{lhs} is decreasing in the variable $b \in [0, a(1-2t_0)]$ and is bounded below by $C_3 \, a(1-2t_0)$.
In conclusion
$$ C_3 \, \, a(1-2t_0) \le  C_2 \, \frac{a(1)-a(1-2t_0)}{2 t_0} $$ which gives \eqref{c1}:
$$ a(1-2t_0) (1-2t_0)^{-(n+1)} \le a(1),$$
and Step 1 is proved.
\qed

\

As we iterate Step 1 we find that the densities of the positivity set in $B_r$, $a(r) r^{-n}$, tend to $0$ as $r=r_0^m \to 0$. After rescaling, it remains to show that if $a(1)$ is sufficiently small, depending on $\gamma$, then $a(1/2)=0$.

\

{\it Step 2:} If $a(1) \le c(\gamma)$ small then for all $r \in [1/2,1]$,
\begin{equation}\label{ddeq}
 a(r-2t)^\delta \le \frac{a(r)-a(r-2t)}{2t}, \quad \quad t=a(r)^\mu,
 \end{equation}
with $\delta$, $\mu$ universal constants.

\

 {\it Proof of Step 2:}
 Assume for simplicity that $r=1$. Notice that by Theorem \ref{P1} it follows that 
$$ u \le C a(1)^\frac{\alpha}{n} \quad \mbox{in $B_1$.}$$
We argue as in Step 1 and improve the last part of the argument. Take
$$ \Psi =\psi (|x|-(1-t_1)) $$
with $t_1 \in (0,t_0]$ such that 
$$\varphi(t_1) = a(1)^\mu \gg \|u \|_{L^\infty(B_1)}.$$
This means that $\{u < \Psi \}$ on $\p B_1$ and now we may take $D= \{u>\Psi \} \cap B_1$. We obtain as above the corresponding inequality \eqref{enin} with $t_0$ replaced by $t_1$. After averaging over the family of translates $\Psi_t$ with $t \in [0,t_1]$ we establish the inequality \eqref{mi} with $t_0$ replaced by $t_1$.
We bound the left hand side as before by taking 
$$ s_1=\varphi(t_1)=a(1)^\mu,$$
and obtain
$$ W(s_1) s_1^{\frac \gamma 2 -1} \, b + (a(1-2t_1)-b)^\frac{n-1}{n} \le C \, \, \frac{a(1)-a(1-2t_1)}{2 t_1}.$$ 
Using that $$s_1= a(1)^\mu \ge a(1-2t_1)^\mu,$$ the coefficient of $b$ in the left hand side is bounded below by a negative power of $a(1-2t_1)$ (provided that $a(1)$ is sufficiently small, depending on $\gamma$). Then, by arguing that 
$$\mbox{either} \quad b \le \frac{a(1-2t_1)}{2} \quad \mbox{ or}  \quad b \ge \frac{a(1-2t_1)}{2}, 
$$we obtain that the left hand side is bounded below by  
$$ a(1-2 t_1)^{1-\delta}, $$
for some $\delta$ universal. After relabeling $\delta$ if necessary we reach the desired discrete differential inequality claimed in Step 2.
$$ a(1-2 s_1)^{1-\delta}  \le \frac{a(1)-a(1-2s_1)}{2 s_1}, \quad \quad s_1=a(1)^\mu.$$
\qed

\

{\it End of the proof:}
Now it is straightforward to check that a nondecreasing function $a(r)$ that satisfies \eqref{ddeq} must vanish when $r=1/2$ if $a(1)$ is sufficiently small. In the continuous setting we obtain $a' \ge a^{1-\delta}$ which implies $$a(r) \le (r-1/2)^M,$$ for some large $M$, provided that the inequality is satisfied at $r=1$. In the discrete setting it follows by induction that the inequality above holds for $r=r_k$ where $r_k$ is the sequence 
$$r_{k+1}=r_k - 2 a(r_k)^\mu, \quad \quad r_0=1.$$

\end{proof}

\begin{rem}\label{R1}
From \eqref{102} and \eqref{1025} it follows that
$$J(u,B_{1/2}) \le J(u,D_1) \le C,$$
with $C$ universal. 
\end{rem}

Next we prove the other side of the density bound using a similar analysis.

\begin{lem}\label{L2}
Let $u$ be a minimizer of $J_\gamma$ in $B_1$ and assume $0\in F(u).$ Then 
 $$|\{u=0\} \cap B_r| \ge c_0 |B_r|.$$
\end{lem}

\begin{proof}
Let $s_0$, $s_1$, and $t_1$ be defined as
$$W(s_0)=1, \quad W(s_1)=M, \quad \varphi(t_1)=s_1,$$
with $M$ a large universal constant to be made precise later. Let $$A_r:=\{ u \le s_1 \} \cap B_r, \quad a(r):=|A_r|.$$

\

{\it Step 1:} We prove that if $a(1) \le c_0$ universal, $M \ge C_0$ and $\gamma$ sufficiently close to $2$ (depending on $M$) then 
\begin{equation}\label{ain}
a(r_0)r_0^{-n}  \le r_0 a(1) \quad \quad \mbox{for some fixed $r_0<1$.}
\end{equation}

\

We first construct a 1D profile.

\begin{lem}\label{lem4}
 There exists a nondecreasing Lipschitz function $\psi:[0,1] \to \R$, with $\psi(0)=0$, which is $C^1$ in the intervals $\{\psi<s_1\}$, $\{\psi>s_1\}$ such that
 
 \
 
 1) $\psi=\varphi$ in $[0,t_1]=\{ \psi \le s_1\}$,
 
 \
 
 2) in $(t_1,1]=\{ \psi>s_1\}$, 
 $$ 2\psi'' - 8n \psi' \ge W'(\psi) $$
 and $\psi$ is constant in $[1/4,1]$,
 
 \
 
 3) $$ \frac 12 W(\psi) \le (\psi')^2 \le  W(\psi) \quad \mbox{ in} \quad [0,t_0]:=\{\psi \le s_0\}.$$ 

\

Here $t_0$ is defined such that $$\psi(t_0)=s_0, \quad \mbox{thus} \quad  W(\psi(t_0))=1.$$

\end{lem}

\

{\it Proof of Step 1:}
Define in $ \overline B_1$ the function $$\Psi(x)=\psi(1-|x|),$$ and denote by $$D:=\{ u< \Psi \}.$$
Notice that $\Psi$ vanishes on $\p B_1$ and coincides with $\varphi(1-|x|)$ near $\p B_1$, hence
\begin{equation}\label{2003}
|\nabla \Psi|=\sqrt{W(\Psi)} \quad \mbox{in} \quad B_1 \setminus B_{1-t_1}=\{\Psi \le s_1\}.
\end{equation}
Also by 2)
$$2 \triangle \Psi \ge W'(\Psi) \quad \quad \mbox{in} \quad \{\Psi > s_1\},$$
and 3) implies
\begin{equation}\label{2000}
 \frac 12 W(\Psi) \le |\nabla \Psi|^2 \le W(\Psi) \quad \mbox{in} \quad B_1 \setminus B_{1-t_0},
\end{equation}
and
\begin{equation}\label{2001}
W(\Psi) \le 1 \quad \mbox{in} \quad B_{1-t_0}.
\end{equation}
Denote by 
$$ D_1:=\{ u > s_1\} \cap D, \quad D_2:= D \setminus D_1,$$
$$F_1:=\{ \Psi > s_1\} \cap D, \quad F_2:= D \setminus F_1.$$

Then $J(u,D) \le J(\Psi,D)$ implies
$$ J(u,D_2) \le J(\Psi,F_2) + J(\Psi,F_1) -J(u,D_1).$$

In $$F_1= D_1 \cup A_{1-t_1}$$ we write $$ \max\{ u, \sigma\}=\Psi-w, \quad \mbox{ with $\Psi \ge w \ge 0$,}$$ and notice that 
$w$ vanishes on $\p F_1$ hence
\begin{align}\label{2006}
\nonumber \int_{D_1} |\nabla u|^2 &= \int_{F_1}|\nabla (\Psi-w)|^2 dx \\
& \ge \int_{F_1} |\nabla \Psi|^2 + 2 w \triangle \Psi dx  \\
\nonumber & \ge \int_{F_1} |\nabla \Psi|^2 +  w W' (\Psi) dx  \\
\nonumber & \ge \int_{F_1} |\nabla \Psi|^2 +  (W (\Psi) - W(\Psi-w))\chi_{D_1} - C W(\Psi)\chi_{A_{1-t_1}} dx,
\end{align}
where in the last inequality we used the convexity of $W$ in $D_1$ and the fact that $W'(\Psi)<0$ in $A_{1-t_1}$ thus
$$wW'(\Psi)  \ge \Psi W'(\Psi) = -\gamma W(\Psi).$$ 
Since $\Psi-w=u$ in $D_1$ we find
$$J(u,D_1) \ge J(\Psi, F_1) - C \int _{A_{1-t_1}} W(\Psi) dx,$$
hence
$$J(u,D_2) \le J(\Psi,F_2) + C \int _{A_{1}} W(\Psi) dx.$$
By Cauchy-Schwartz and co-area formula we obtain
$$J(u,D_2) \ge \int_0^{s_1} \mathcal H^{n-1} (\{u=s \} \cap D \} \sqrt{W(s)} ds,$$
while, by \eqref{2003},
$$J(\Psi,F_2) = \int_0^{s_1} \mathcal H^{n-1} (\{\Psi=s \} \cap D \} \sqrt{W(s)} ds.$$
Hence
$$J(\Psi,F_2) \le \int_0^{s_1} \mathcal H^{n-1} (\{\Psi=s \} \cap A_1 \} \sqrt{W(s)} ds,$$
and we also write
 $$\int_{A_1} W(\Psi) dx =\int_{A_1 \cap B_{1-t_0}} W(\Psi) dx + \int_{A_1 \setminus B_{1-t_0}} W(\Psi) dx.$$
 By \eqref{2001} the second term is bounded by $ |A_1|$, while by \eqref{2000} and the co-area formula as above,  the first integral is bounded by
 $$ C \int_0^{s_0} \mathcal H^{n-1} (\{\Psi=s \} \cap A_1 \} \sqrt{W(s)} ds. $$
Using that
$$E:=\{u=0\} \cap B_{1-t_1} \subset \{ u \le s \le \Psi \}, \quad s \in [0,s_1],$$ we find
by the isoperimetric inequality that
$$|E|^\frac{n-1}{n} \int_0^{s_1} \sqrt{W(s)} ds  \le J(u,D_2)+ J(\Psi,F_2).$$
Notice that as $\gamma \to 2$ (and fixed $M$), the integral converges to $$\int_0^1 \sqrt {W(s)} ds = \frac 12.$$ Also
$$W(s_1) |A_{1-t_1} \setminus E| \le \int_{A_{1-t_1}}W(u) dx \le J(u,D_2)$$
 In conclusion
 \begin{align}\label{2004}
 \nonumber \frac 14|E|^\frac{n-1}{n} + M & |A_{1-t_1} \setminus E| \le \\
 & \le C\int_0^{s_0} \mathcal H^{n-1} (\{\Psi=s \} \cap A_1 \} \sqrt{W(s)} ds + C |A_1|.
 \end{align}
 Since $|E| \le a(1)\le c_0$ is sufficiently small, and $M \ge C_0$, the left hand side is bounded below by
$$ \frac{C_0}{2} |A_{1-t_1}| \ge \frac{C_0}{2} a(1-2 t_0).$$
We average the right hand side by taking as test functions 
$$\Psi_t(x)=\psi(1-t-|x|),\quad \quad t \in [0,t_0],$$ and obtain
 $$\frac{C_0}{2} a(1-2 t_0) \le C \frac{a(1)-a(1-2t_0)}{2t_0}  + C a(1)$$
 which implies the desired conclusion \eqref{ain} with $r_0=1-2t_0$,
 $$ a(1-2 t_0) (1-2t_0)^{-(n+1)} \le a(1),$$
 provided $C_0$ is chosen sufficiently large.
 
 \qed

Next we prove the lemma when $\gamma$ is close to 2.
 
\

{\it Step 2:} If $\gamma$ is sufficiently close to 2 then $|\{u=0\} \cap B_1| \ge c_0/2$.

\

{\it Proof of Step 2:}
If the conclusion does not hold then
\begin{equation}\label{2005}
|\{u=0\} \cap B_1| \le c_0/2 \quad \Longrightarrow \quad a(1) \le c_0.
\end{equation}
Indeed, otherwise $$ |\{ 0<u\le s_1\} \cap B_1|\ge c_0/2,$$ and we can apply inequality \eqref{2004} (with $A_{1-t_1}$, $A_1$ replaced by $A_1$, respectively $A_{1+t_1}$) and obtain
$$ M \frac{c_0}{2} \le C\int_0^{s_0} \mathcal H^{n-1} (\{\Psi=s \} \cap A_{1+t_1} \} \sqrt{W(s)} ds + C |A_{1+t_1}| \le C.$$
We get a contradiction by choosing $M$ universal, sufficiently large, and \eqref{2005} is proved.
Now we may apply Step 1 and obtain
$$ |\{ u \le r_0^\alpha s_1\} \cap B_{r_0}| r_0^{-n} \le a(r_0) r_0^{-n} \le r_0 a(1),$$
which can be rescaled and iterated indefinitely. Thus, after a rescaling of $u$ of factor $r_0^m$ with $m$ large we find that 
$a(1)$ can be made arbitrarily small.

We reached a contradiction to $0 \in F(u)$ since, by Theorem \ref{P1}, 
$$a(1) \ge c(s_1) >0.$$  
\qed

Finally, we prove the conclusion also when $\gamma$ stays away from 2.

\

{\it Step 3:} If $\gamma \le 2-\delta$ then  $|\{u>0\} \cap B_1| \ge c(\delta)$. 

\

{\it Proof of Step 3:}
This follows easily by compactness. However, here we sketch a direct proof that follows from an argument in Step 1. 

First we claim that 
$$\max_{\p B_1} u \ge c(\delta),$$
for some $c(\delta)>0$ small. Otherwise, the energy of $u$ in $B_{1/2}$ is sufficiently small, which implies that $\{u>0\}$ has small measure in $B_{1/2}$ and contradicts Lemma \ref{L1}. 

Next, let $v$ be the solution to the Euler-Lagrange equation $2 \triangle \Psi=W'(\Psi)$ in $B_1$, $v=u$ on $\p B_1$. Since $v$ is superharmonic, $v(0)>c(\delta)$. Moreover, $W(v)$ is bounded by an integrable function in $B_1$. As in Step 1, the inequality
$$ J(u,B_1) \le J(v,B_1) $$
implies (see \eqref{2006} with $s_1=0$, $D_1=F_1=B_1$),
$$ \int_{B_1} |\nabla (v-u)|^2 dx \le C \int_{\{u=0\}}W(v)dx.$$
The left hand side is bounded below by a $c_1(\delta)$ which follows from Theorem \ref{P1} and $(v-u)(0)=v(0) \ge c(\delta)$. This shows that $\{u=0\}$ cannot have arbitrarily small measure.
 \end{proof}

 It remains to prove the existence of the 1D profile of Lemma \ref{lem4}.

  \
 
 {\it Proof of Lemma \ref{lem4}:} We construct $\psi$ by defining its corresponding function $g$ as in $\eqref{gred}$, $\eqref{inv}$. Let $g$ be the perturbation of $W$
 $$ g(s)= W(s) + \left(-\frac 12 + C_n(s^{1-\frac \gamma 2}-s_1^{1-\frac \gamma 2})\right) \chi_{[s_1,1]},$$
 with $C_n=8n$. Let $s_2$ be defined as 
 $$W(s_2)=\frac 14,$$
 hence $s_2=4^{1/\gamma} s_0 \sim s_0$, and notice that $s_2 \to 0$ as $\gamma \to 2$. Moreover $$W' =- \gamma W/s \le -C \quad \mbox{ in} \quad [0,s_2]$$ which implies that $g' \le -C$ in the same interval.
Furthermorer, for $\gamma$ sufficiently close to $2$ (depending on $M$), then $s_1^{1-\gamma/2}$ is close to $1$ hence the error $g(s)-W(s)$ is uniformly close to the constant $-1/2$ in the interval $[s_1,1]$. 

These facts imply that $ g \le W$, and $g$ crosses $0$ at some point $\sigma \in [s_0,s_2]$, and 
 $$ g \ge \frac 12 W \quad \mbox{in} \quad [s_1,s_0], $$
 which gives property 3). Property 1) follows directly from the definition.
 Finally, property 2) holds since in $(s_1,1] \cap \{g>0\}$
 $$g' - W' = C_n (1-\frac \gamma 2) s^{-\gamma/2} \ge 8n \sqrt W \ge 8n \sqrt g.$$
Moreover,
 \begin{align*}
 \int_{\{g>0\}} (2 g)^{-1/2} ds &= \int_0^\sigma (2 g)^{-1/2} ds \\
 & \le \int_0^{s_0} W^{-1/2} ds + C\int_{s_0}^\sigma s_0^{1/2} (\sigma-s)^{-1/2} ds \\
 & \le 2^{1/2} t_0 + C s_0 \\
 &\le 1/4
 \end{align*}
 which shows that $\psi$ is constant outside an interval of length $1/4$.
 \qed
 
 We conclude this section with a proof of Corollary \ref{C0}.
 
 \
 
  {\it Proof of Corollary \ref{C0}.} 
  Assume that $u$ is a minimizer of $J$ in $B_2$ and $0 \in F(u)$. First we prove that
  \begin{equation}\label{1070}
  c \le J(u,B_1) \le C,
  \end{equation}
  with $c$, $C$ universal constants. 
  
  The upper bound follows from Remark \ref{R1}. For the lower bound, we use that 
  $$(1-c_0) |B_1| \ge |\{u>0\} \cap B_1| \ge c_0 |B_1|.$$ 
  Let $s_0$ be defined as in the proof of Lemma \ref{L1}, see \eqref{s_0}. If
 \begin{equation}\label{1071}
  |\{u >s_0\} \cap B_1| \le \frac{c_0}{2} |B_1|,
  \end{equation}
  then 
  $$ | \{0<u \le s_0\} \cap B_1| \ge \frac{c_0}{2} |B_1|. $$
  
  In this last set $W(u) \ge W(s_0)=1$, and the lower bound is obtained from the potential term.
  
  On the other hand, if the opposite inequality in \eqref{1071} holds, then for all $s \in (0,s_0)$
  the density of $\{u>s\}$ in $B_1$ is bounded both above and below by universal constants. Now the lower bound follows from \eqref{1023} and the Poincar\'e inequality for $\chi_{\{u>s\}}$ in $B_1$.
  
 The existence of a full ball of radius $c'$ included in $\{u>0\} \cap B_1$ (or $\{u=0\}\cap B_1$) follows by a standard covering argument. We sketch it below.
 
 We take a collection of $m$ disjoint balls $B_\rho(x_i)$, $x_i \in \{u>0\} \cap B_1$ such that $\cup B_{5\rho}(x_i)$ covers $\{u>0\} \cap B_1$. It follows that $m \sim \rho^{-n}$. If we assume that each $B_{\rho/2}(x_i)$ intersects the free boundary then, by the rescaled version of \eqref{1070}, 
 $$J(u, B_\rho(x_i)) \ge c \rho^{n-\alpha \gamma}.$$
We obtain 
$$J(u,B_1) \ge m \, c \, \rho^{n-\alpha \gamma},$$ 
and we contradict the upper bound if $\rho$ is chosen small, universal.
\qed
 
 \section{The Gamma convergence}
 
 In this section we prove our main result Theorem \ref{TM}. We start by constructing an interpolation between two functions which are close to each other in a ring.
 
 \begin{prop}\label{P2}
 Let $u_k,v_k$ be sequences in $H^1(B_1)$ and $\gamma_k \to 2^-$. Assume that for some $\rho \in (\frac 12,1)$ and $\delta>0$ small,
 $$J_{\gamma_k}(u_k, B_{\rho+\delta}), \quad J_{\gamma_k}(v_k,B_{\rho+ \delta})$$
 are uniformly bounded, and
 $$\| u_k-v_k\|_{L^2} + \|u_k^{1-\frac{\gamma_k}{2}}-v_k^{1-\frac{\gamma_k}{2}}\|_{L^1} \to 0 \quad \mbox{ in} \quad  B_{\rho+\delta} \setminus \bar B_\rho, \quad \mbox{as $k\to \infty.$ }$$ 
 Then, there exists $w_k \in H^1(B_1)$ with $$w_k:= \begin{cases}v_k \quad \text{in $B_{\rho}$}\\ u_k \quad \text{in $B_1 \setminus \bar B_{\rho+\delta}$}\end{cases}$$ such that 
 $$J_{\gamma_k}(w_k, B_1) \leq J_{\gamma_k}(v_k, B_{\rho + \delta}) + J_{\gamma_k}(u_k, B_1 \setminus \bar B_{\rho}) + o(1),$$
 with $o(1) \to 0$ as $ k \to \infty$.

 \end{prop}
 
 \begin{proof}
Fix $\eps >0$ small. We prove the conclusion with $o(1)$ replaced by $C \eps$ for some $C$ universal. 
Since the energies of $u_k$ and $v_k$ are uniformly bounded, we can decompose the annulus $B_{\rho+ \delta}\setminus B_\rho$ into a disjoint union of $\sim \eps^{-1}$ annuli, and after relabeling $\rho$ and $\delta$ we may assume that
$$J_{\gamma_k}(u_k, B_{\rho+\delta}\setminus B_\rho) \le \eps, \quad J_{\gamma_k}(v_k,B_{\rho+ \delta}\setminus B_\rho) \le \eps.$$
 For simplicity of notation we drop the subindex $k$. 
 
  First we prove the result under the additional assumption 
  \begin{equation}\label{3000}
 \mbox{ $u \ge v$ in $B_{\rho+\delta} \setminus B_\rho$. }
 \end{equation}
 
 Denote by
 $$\psi_r(x)= \varphi (|x|-r), r \in [\rho, \rho + \frac \delta 4],$$
 and let 
 $$\Psi_r = \min \{ u, \max\{\psi_r,v \} \}.$$
 Notice that $$ u \ge \Psi_r \ge v \quad \mbox{in $B_1$, and} \quad  \quad \Psi_r=v \quad \mbox{in $B_{\rho}$.} $$
 Let $$D_r:=\{u> \Psi_r>v \} \cap B_{\rho+\delta},$$
 then, by the property \eqref{1Dp2} of the one-dimensional solution $\varphi$, we find
 \begin{equation}\label{3001}
 J(\Psi_r, D_r)= J(\psi_r, D_r)=\int_0^ 1 \mathcal H^{n-1} (\{\Psi_r =s\} \cap D_r) \, 2 \sqrt{W(s)}ds.
 \end{equation}
 Notice that
$$ \{\Psi_r =s\} \cap D_r = \{u>s>v\}\cap \p B_{r+\varphi^{-1}(s)} \cap B_{\rho+\delta}.$$
Thus, we average \eqref{3001} for $r \in [\rho, \rho+\delta/4],$ and obtain
\begin{align}\label{3002}
\nonumber \fint_{\rho}^{ \rho+\delta/4} J&(\Psi_r,  D_r) dr \le \\
 & \frac{C} {\delta} \int_0^1 \mathcal H^n\left((\{u>s>v\}) \cap (B_{\rho+\delta}\setminus B_\rho) \right) \,  2\sqrt{W(s)}ds.
 \end{align}
 We use \eqref{1Dp0} and the change of coordinates $$s^{1-\gamma/2}= \sigma \quad \mbox{ and obtain} \quad 2\sqrt{W(s)} ds = d \sigma.$$
 The right hand side in \eqref{3002} equals
 \begin{align*}  \frac{C} {\delta} \int_0^1 \mathcal H^n& \left (  \{u^{1-\gamma/2}>\sigma >v^{1-\gamma/2}\}  \cap (B_{\rho+\delta}\setminus B_{\rho}) \right ) d \sigma \\
   \le  & \frac{C} {\delta}\left  \|u^{1-\gamma/2}-v^{1-\gamma/2}\right \|_{L^1(B_{\rho+\delta}\setminus B_{\rho})}.
 \end{align*}
 
  Thus, for all $k$ sufficiently large, we can find an $r=r_k \in [\rho, \rho+\delta/4],$  such that
  $$ J(\Psi_r,D_r) \le \eps.$$
  Since in the annulus $B_{\rho+\delta}\setminus B_{\rho}$ the function $\Psi_r$ coincides with $u$ or $v$ outside  $D_r$ we find
  \begin{equation}\label{3003}
  J(\Psi_r,B_{\rho+\delta}\setminus B_{\rho}) \le 3 \eps. 
  \end{equation}
  Finally we define $$ w= \eta \Psi_r + (1-\eta) u,$$
  with $\eta \in C_0^\infty(B_{\rho + \delta}) $ a cutoff function with $\eta=1$ in $B_{\rho+\delta/2}$. Clearly, $w=u$ outside $B_{\rho+\delta}$ and $w=\Psi_r$ in $B_{\rho+\delta /2}$, hence $w=v$ in $B_\rho$. Moreover, 
  $$u \ge w \ge \Psi_r>0 \quad \Longrightarrow \quad W(w) \le W(\Psi_r) \quad \mbox{in $B_{\rho+\delta}\setminus B_{\rho + \delta/2}$}.$$ Since
  $$|\nabla w|^2 \le 3\left(|\nabla \Psi_r|^2 + |\nabla u|^2 + |\nabla \eta|^2(u-\Psi_r)^2 \right),$$
  we find 
  $$ J(w,B_{\rho+\delta}\setminus B_{\rho}) \le 3 \left(J(\Psi_r,B_{\rho+\delta}\setminus B_{\rho}) + J(u,B_{\rho+\delta}\setminus B_{\rho}) + C(\delta) \| \Psi_r-u\|^2_{L^2} \right).$$
  Using that, 
  $$|u-\Psi_r| \le |u-v|,$$
  we obtain $$C(\delta) \| \Psi_r-u\|^2_{L^2} \to 0 \quad \mbox{as $k \to \infty$}.$$
  We find $$J(w,B_{\rho+\delta}\setminus B_{\rho}) \le 15 \eps,$$ for all large $k$,  which gives the desired conclusion under the assumption \eqref{3000}.
  
  \
  
The general case follows easily from the interpolation procedure between the two ordered functions described above. We apply it two times, first in the annulus $B_{\rho + \delta} \setminus B_{\rho + \delta/2}$ where we interpolate between $u$ and $\min\{u,v\}$ and then in the annulus $B_{\rho + \delta/2} \setminus B_{\rho}$ where we interpolate between $\min\{u,v\}$ and $v$.
 \end{proof}

We recall now the functional $\mathcal F$ introduced in Section 2, which is defined on the space of pairs $(u,E) \in \mathcal A(\Omega)$ 
$$ \mathcal A(\Omega):=\{(u,E)| \quad u \in H^1(\Omega),\quad \mbox{$E$ Caccioppoli set, $u \ge 0$ in $\Omega$, $u=0$ a.e. in $E$} \},$$
given by the Dirichlet - perimeter energy
$$\mathcal F_\Omega(u,E)= \int_\Omega |\nabla u|^2 dx + P_\Omega(E).$$
Here $P_{\Omega}(E)$ represents the perimeter of $E$ in $\Omega$
$$ P_\Omega(E)=[\nabla \chi_E]_{BV(\Omega)}=\int_\Omega |\nabla \chi_E|.$$
 
In the next two lemmas we establish the $\Gamma$-convergence of the $J_\gamma$ to $\mathcal F$.
 \begin{lem}[Lower semicontinuity]\label{L3.5}
 Let $\gamma_k \to 2^-$ and $u_k$ satisfy
  $$ u_k^{1-\gamma_k/2} \to \chi_{E^c} \quad \mbox{in $L^1(\Omega)$}, \quad u_k \to u \quad \mbox{in $L^2(\Omega)$}.$$
  Then
  $$\liminf J_{\gamma_k}(u_k,\Omega) \ge \mathcal F_{\Omega}(u,E).$$

  \end{lem}
\begin{proof} After passing to a subsequence we may assume that the  two convergences above hold pointwise a.e. in $\Omega$. This implies that $ \{u>0\} \setminus E^c$ is a set of measure zero, hence $u=0$ a.e. on $E$, and $(u,E)$ is an admissible pair.

We write
$$J_{\gamma_k}(u_k,\Omega) = J_{\gamma_k}(u_k,\Omega \cap \{u_k \le \eps \}) +  J_{\gamma_k}(u_k,\Omega \cap \{u_k > \eps \}).$$
By the coarea formula and the definition of $W$ (see \eqref{Jf})
\begin{align}\label{300co}
\nonumber J_{\gamma_k}(u_k,\Omega \cap \{u_k \le \eps \}) & \ge \int_{\{u_k \le \eps\}} |\nabla u_k| 2 \sqrt{W(u_k)} dx \\
\nonumber &=\int_{\{u_k \le \eps\}} |\nabla u_k^{1-\gamma_k/2}| dx\\
&=\int_{\Omega} |\nabla \overline u_k^{1-\gamma_k/2}| dx, \quad \mbox{with} \quad \overline u_k:= \min\{u_k, \eps\}.
\end{align}
Moreover, $\overline u_k^{1-\gamma_k/2}$ converges in $L^1$ to $\chi_E$, hence
$$\liminf J_{\gamma_k}(u_k,\Omega \cap \{u_k \le \eps \}) \ge \int_{\Omega} |\nabla \chi_ E| dx,$$
by the lower semicontinuity of the BV norm.
On the other hand 
$$J_{\gamma_k}(u_k,\Omega \cap \{u_k > \eps \}) \ge \int_\Omega |\nabla (u_k-\eps)^+|^2 dx,$$
and since $(u_k -\eps)^+ \to (u-\eps)^+$ in $L^2$, we obtain
$$\liminf J_{\gamma_k}(u_k,\Omega \cap \{u_k > \eps \}) \ge \int_{\Omega} |\nabla(u-\eps)^+|^2 dx.$$
By adding the inequalities we find
$$\liminf J_{\gamma_k}(u_k,\Omega) \ge \int_{\Omega} |\nabla(u-\eps)^+|^2 dx + P_\Omega (E),$$
and the conclusion is proved by letting $\eps \to 0$.
\end{proof}

 \begin{lem}\label{L4}
 Let $(u,E) \in \mathcal A(\Omega)$ with $u$ a continuous function in a Lipschitz domain $\overline \Omega$. Then, given a sequence $\gamma_k \to 2^-$ we can construct a sequence $u_k$ such that
 $$ u_k^{1-\gamma_k/2} \to \chi_{E^c} \quad \mbox{in $L^1(\Omega)$}, \quad u_k \to u \quad \mbox{in $L^2(\Omega)$},$$
 $$ J_{\gamma_k}(u_k,\Omega) \to  \mathcal F_{\Omega}(u,E).$$
 \end{lem}
 
In view of the lower semicontinuity property in $\Omega \setminus \overline D$, where $D \subset \Omega$ is a subdomain, we obtain that
 $$\int_{\overline D} |\nabla u|^2 dx + \int_{\overline D} |\nabla \chi_E| \ge \limsup   J_{\gamma_k}(u_k, D).$$
 \begin{proof}
 For the convergence of the energies it suffices to show that
 $$\limsup J_{\gamma_k}(u_k,\Omega) \le  \mathcal F_{\Omega}(u,E).$$
 Fix $\eps >0$ small. First we approximate $E$ in $\Omega$ by a smooth set $F \subset \R^n$ which is included in the open set $\{u< \eps\}$ in $\Omega$ (which contains a neighborhood of $E$). Precisely, by the results of Modica \cite{M}, there exists a smooth set $F$ which approximates $E$ in $\Omega$ in the sense that
 $$ F \cap \Omega \subset \{ u < \eps \} ,$$
 $$ \|\chi_{F \cap \Omega} - \chi_E\|_{L^1} \le \eps, \quad \quad P_{B_1}(F \cap \Omega) \le P_{\Omega}(E) + \eps,$$
  \begin{equation}\label{3005}\mathcal H^{n-1}(\partial F \cap \p \Omega) =0.
  \end{equation}
In view of this, it suffices to prove the lemma with $E$ replaced by $\tilde E:=F$ and $u$ replaced by $\tilde u:=(u-2\eps)^+$ which is an approximation of $u$ in $H^1(\Omega)$. Notice that by construction $\tilde u$ vanishes in a $\delta$-neighborhood of $\tilde E$ for some small $\delta$. We define $u_k$ in $B_1$ as
$$ u_k:= \max\{ \varphi_k(d), \tilde u\},$$
where $d$ represents the distance in $\R^n$ to $\tilde E$. Next we check that $u_k$ satisfies the desired conclusions.

Clearly $u_k=0$ on $\tilde E$, and using that $$C \ge u_k \ge \varphi_k(d) \quad \mbox{ on} \quad  \Omega \setminus \tilde E,$$ and $1-\gamma_k/2 \to 0^+$ we have 
$$u_k^{1-\gamma_k/2} \to 1 \quad \mbox{in} \quad \Omega \setminus \tilde E,$$
hence
$$ u_k^{1-\gamma_k/2} \to \chi_{\tilde E} \quad \mbox{in $L^1(\Omega)$.}$$
Here we used that  $\varphi_k(d)=c_\gamma^* d^\alpha$, with $c_\gamma^*$ defined in \eqref{1Dp}, and we have
$$(c_\gamma^*)^{1-\gamma/2}  \to 1 \quad \mbox{as $\gamma \to 2$.}$$
Since $\varphi_k(d)$ converges uniformly to $0$ as $k \to \infty$ we also obtain
$$u_k \to \tilde u \quad \mbox{in} \quad L^2(\Omega).$$ 
Using property \eqref{1Dp2} we obtain that
$$ J\left(\varphi(d), \{a<d<b\}\cap \Omega\right) = \int_{\varphi(a)}^{\varphi(b)}  \mathcal H^{n-1}(\{ \varphi(d)=s\} \cap \Omega) 2 \sqrt{W(s)} ds$$
$$=\int_a^b \mathcal H^{n-1}(\{ d=t\} \cap \Omega) \omega_\gamma(t) dt,$$
with
$$\omega_\gamma(t) :=2 \sqrt {W(\varphi(t))} \varphi'(t).$$
Notice that $$\omega_\gamma(t) dt = \varphi'(t) ^2 + W(\varphi(t)) \, dt,$$ represents the measure of the one-dimensional solution which, as $k \to \infty$, converges weakly in any bounded interval $[-a,a]$ to the Dirac delta measure at $0$.
On the other hand \eqref{3005} implies that 
$$ \mathcal H^{n-1}(\{ \varphi(d)=t\} \cap \Omega) \to P_{\Omega}(\tilde E \cap \Omega) \quad \mbox{as $t \to 0$}.$$
In conclusion, we find that as $k \to \infty$
$$ J\left(\varphi(d), \{0<d<\delta\}\cap \Omega\right) \to P_{\Omega}(\tilde E \cap \Omega)$$
and
$$ J\left(\varphi(d), \{d>\delta\}\cap \Omega \right) \to 0.$$
Using that 
$$u_k=\varphi(d) \quad \mbox{if $d < \delta$,}$$
and
$$u_k \ge \varphi(d) \quad \Longrightarrow \quad W(u_k) \le W(\varphi(d)) \quad \mbox{if $d > \delta$,}$$ 
we find 
$$ \limsup J(u_k, \Omega) \le P_{\Omega}(\tilde E \cap \Omega) + \int_{\Omega} |\nabla \tilde u|^2 dx. $$
 \end{proof}
 
 
 We are finally ready to prove our main theorem.
 
 \

 \begin{proof}[Proof of Theorem \ref{TM}]
 The $L^2$ convergence follows from the uniform bound of the $u_k$ in $H^1(\Omega)$. 

By the coarea formula (see \eqref{300co}) we find that
$$[u_k^{1-\gamma_k/2}]_{BV(\Omega)} \le M$$
and using the inequality
$$ u_k^{1-\gamma_k/2} \le 1 + u_k^2$$
we find that $u_k^{1-\gamma_k/2} $ are uniformly bounded in $BV (\Omega)$. Thus, after passing to a subsequence, we have
\begin{equation}\label{3059}
u_k^{1-\gamma_k/2}  \to g \quad \mbox{in $L^1(\Omega)$,}  
\end{equation}
for some non-negative $g \in BV(\Omega)$. We claim that 
\begin{equation}\label{3006}
\mbox{$g=\chi_{E^c}$ for some set $E$.}
\end{equation}

First we show that for all $ \delta>0$ small $$\{\delta \le g \le 1-\delta\} \quad \mbox{has measure zero}.$$
Otherwise, for all large $k$, the set $$\{ \delta/2 \le u_k^{1-\gamma_k/2} \le 1-\delta/2 \} $$
has measure bounded below by a fixed positive constant. On this set
\begin{equation}\label{3007}
W(u_k) \ge W ((1-\delta/2)^\frac{2}{2-\gamma_k}) =c_{\gamma_k} (1-\delta/2)^{-\frac{2\gamma_k}{2-\gamma_k}} \to \infty,
\end{equation}
 as $k \to \infty$ and we contradict the uniform upper bound for the energy of $u_k$ in $\Omega$.
 
 Similarly we find that the set 
 $$\{g \ge 1 + \delta\} \quad \mbox{has measure zero}.$$
Indeed, otherwise 
$$ \{u_k^{1-\gamma_k/2} \ge 1 + \delta/2 \} $$
has measure bounded below by a fixed positive constant. Then we contradict the uniform upper bound for the $L^2$ norm of $u_k$ since on the set above
$$ u_k^2 \ge (1+ \delta/2)^\frac{4}{2-\gamma_k} \to \infty $$
as $k \to \infty$, and the claim \eqref{3006} is proved.
 
 The argument above implies also that 
 \begin{equation} \label{3058}
  \chi_{\{u_k>0\}} \to \chi_{E^c} \quad \mbox{in $L^1(\Omega)$}.
  \end{equation}
  For example if 
  $$ |\{u_k>0\} \setminus E^c| \ge \mu>0 $$
  for some positive constant $\mu$ independent of $k$, then \eqref{3059}-\eqref{3006} imply
  $$|\{0 < u_k^{1-\gamma_k/2} \le \frac 12\}| \ge \mu/2,$$
  and we get a contradiction as in \eqref{3007}.
Also  
$$|E^c \setminus\{u_k>0\}|=|E^c \cap \{u_k=0\}| \to 0, $$
as $k \to \infty$, follows from the convergence \eqref{3059}-\eqref{3006}.

Next we assume that $u_k$ are minimizers for $J_{\gamma_k}$ and prove the minimality of $(u,E)$ for $\mathcal F$. The argument is standard and follows from Proposition \ref{P2}. We sketch it for completeness. 

For simplicity let $\Omega=B_1$. Since the functions $u_k$ are uniformly H\"older continuous on compact sets of $B_1$ we find that the limiting function $u$ is H\"older continuous in $B_1$ and the convergence $u_k \to u$ is uniform on compact subsets. 

Let $(v,F)$ be an admissible pair which coincides with $(u,E)$ near $\p B_1$ and let 
$$\mathcal R:=B_{\rho + \delta} \setminus B_\rho,$$
be an annulus near $\p B_1$ where the two pairs coincide.

Denote by $v_k$ be the functions constructed in Lemma \ref{L4} corresponding to the pair $(v,F)$ in $B_{\rho+\delta}$. Since $u_k$ and $v_k$ satisfy the hypotheses of Proposition \ref{P2} we can construct $w_k$ as the interpolation between $u_k$ and $v_k$. By the minimality of $u_k$ in $B_1$ and the conclusion of Proposition \ref{P2} we have
$$J(u_k,B_1) \le J(w_k,B_1) \le J(u_k, B_1 \setminus B_\rho) + J(v_k, B_{\rho+ \delta}) + o(1).$$
This gives
$$ J(u_k,B_\rho) \le J(v_k, B_{\rho+ \delta}) + o(1),$$
and by taking $k \to \infty$, we find from Lemmas \ref{L3.5} and \ref{L4}
$$ \mathcal F_{B_\rho}(u,E) \le \mathcal  F_{B_{\rho+\delta}}(v,F).$$
We let $\rho \to 1$ and obtain the desired conclusion
$$\mathcal F_{B_1}(u,E) \le \mathcal F_{B_1}(v,F).$$
Finally, the uniform convergence of the free boundaries follows from the uniform density estimates and the $L^1$ convergence \eqref{3058}.
 \end{proof}

\end{document}